\documentclass[12pt]{article}
\usepackage{graphicx,color}
\usepackage{amsmath,amsfonts,amssymb,amsthm}
\usepackage{rsfso}
\usepackage{ulem}
\usepackage{comment}

\usepackage[numbers]{natbib}
\definecolor{darkgreen}{rgb}{0.06, 0.56, 0.2}

\newcommand{\vc}[1]{{\boldsymbol #1}}

\newcommand{\sr}[1]{{\mathcal #1}}
\newcommand{\dd}[1]{\mathbb{#1}}

\newcommand{\br}[1]{\langle #1 \rangle}
\newcommand{\ol}{\overline}

\newcommand{\eq}[1]{(\ref{eq:#1})}
\newcommand{\lem}[1]{Lemma~\ref{lem:#1}}

\newcommand{\thr}[1]{Theorem~\ref{thr:#1}}

\newcommand{\pro}[1]{Proposition~\ref{pro:#1}}

\newcommand{\dfn}[1]{Definition~\ref{dfn:#1}}
\newcommand{\ass}[1]{Assumption~\ref{ass:#1}}
\newcommand{\rem}[1]{Remark~\ref{rem:#1}}

\newcommand{\app}[1]{Appendix~\ref{app:#1}}
\newcommand{\sectn}[1]{Section~\ref{sec:#1}}

\newcommand{\lemt}[1]{\ref{lem:#1}}

\newcommand{\pend}{\hfill \thicklines \framebox(6.6,6.6)[l]{}}

\newenvironment{proof*}[1]{\noindent {\sc  #1} \rm}{\pend}

\renewcommand{\theequation}{\thesection.\arabic{equation}}

\newtheorem{theorem}{Theorem}[section]
\newtheorem{lemma}{Lemma}[section]
\newtheorem{assumption}{Assumption}[section]
\newtheorem{proposition}{Proposition}[section]

\newtheorem{remark}{Remark}[section]

\newtheorem{definition}{Definition}[section]

\newcommand{\setnewcounter} {
\setcounter{subsection}{0}
\setcounter{equation}{0}
\setcounter{conjecture}{0}
\setcounter{assumption}{0}
\setcounter{question}{0}
\setcounter{definition}{0}
\setcounter{theorem}{0}
\setcounter{corollary}{0}
\setcounter{lemma}{0}
\setcounter{proposition}{0}
\setcounter{remark}{0}
}

\begin{document}
 \title{\Large \bf Stability of a cascade system with two stations and its extension for multiple stations}

\author{Masakiyo Miyazawa\footnotemark \and Evsey Morozov\footnotemark}
\date{July 9, 2023}

\maketitle

\begin{abstract}
We consider a two-station cascade system in which waiting or externally arriving customers at station $1$ move to the station $2$ if the queue size of station $1$ including an arriving customer itself and a customer being served is greater than a given threshold level $c_{1} \ge 1$ and if station $2$ is empty. Assuming that external arrivals are subject to independent renewal processes satisfying certain regularity conditions and service times are $i.i.d.$ at each station, we derive necessary and sufficient conditions for a Markov process describing this system to be positive recurrent in the sense of Harris. This result is extended to the cascade system with a general number $k$ of stations in series. This extension requires certain traffic intensities of stations $2,3,\ldots, k-1$ for $k \ge 3$ to be defined. We finally note that the modeling assumptions on the renewal arrivals and $i.i.d.$ service times are not essential if the notion of the stability is replaced by a certain sample path condition. This stability notion is identical with the standard stability if the whole system is described by the Markov process which is a Harris irreducible $T$-process. 
\end{abstract}

\footnotetext[1]{Department of Information Sciences,
Tokyo University of Science, Noda, Chiba, Japan \&
	School of Data Science, Chinese University of Hong Kong, Shenzhen, China.}
\footnotetext[2]{Institute of Applied Mathematical Research of Karelian Research Centre, Russian Academy of Sciences, Petrozavodsk, Karelia, Russia; Petrozavodsk State University, Petrozavodsk, 185910; Moscow Center for Fundamental and Applied Mathematics, Moscow State University, Moscow 119991, Russia.}

\section{Introduction}
\label{sec:introduction}
\setnewcounter
We are interested in the stability of a service system with two single server stations, numbered as $1,2$, where the stability means that a Markov process describing this system is positive recurrent in the sense of Harris (see \dfn{Harris 1}). We assume the following system dynamics for this queueing model. Each station has renewal arrivals and $i.i.d.$ service times. Both stations have single servers, which are independently working in parallel, but waiting or exogenously just arriving customers at station $1$ moves to the station $2$ if the queue size of station $1$ including a customer being served is greater than a given threshold level $c_{1} \ge 1$ and if station $2$ is empty, where an exogenously arriving customer is counted in the queue. The customer arriving from station 1 immediately starts its service at station 2, and gets service only when no other customer is in station 2. 
We refer to this queuing system as a 2-station cascade system. Since this system naturally arises in practice but its analysis is not easy, its stability is attracting researchers of queueing systems

For station $i=1,2$, let $\lambda_{i}$ be the arrival rate of exogenously arriving customers, and let $\mu_{i}$ is the service rate of those customers. Let $\rho_{i} = \lambda_{i}/\mu_{i}$. Since some of arriving customers at station 1 may move to station 2, $\rho_{1}$ is not an actual traffic intensity (the mean amount of service processed per unit time), so it is nominal. We define $\rho^{*}_{1}(\xi)$ as the time average of the probability that the station $1$ is not empty for an arbitrarily given initial distribution $\xi_{	}$ of the system at time $0$, which is formally defined as \eq{rho* 1} in \sectn{two station}. We refer to $\rho^{*}_{1}(\xi)$ as an effective traffic intensity at station $1$ given $\xi$.

In this setting, it is expected that the system is stable if and only if $\rho^{*}_{1}(\xi) < 1$ for some $\xi$ and $\rho_{2} < 1$. We prove this characterization of the stability in our framework (see \lem{stability 2}), but it is less tractable because $\rho^{*}_{1}(\xi)$ is hard to compute. In the literasture, a computable condition has been studied (e.g., see \cite{MoroStey2013}), but the stability is not fully answered. For the cascade system with more than two stations, necessary and sufficient conditions are separately considered in \cite{DelgMoro2014}, but there are gaps between those conditions.

In this paper, we solve this stability problem for the 2-station cascade system, assuming certain regularity conditions on the inter-arrival time distributions of exogenous customers at each station. This extends the known results. For example,  \citet{MoroStey2013} obtain tractable sufficient conditions for the stability under the extra assumption that the 1st station has Poisson arrivals. We show that, not assuming Poisson arrivals, those conditions are necessary and sufficient. We further extend this result for a $k$-station cascade system for a general integer $k \ge 3$ using the effective traffic intensities $\rho^{*}_{1}(\xi)$ for $i=2,3,\ldots, k-1$. As those traffic intensities are hard to compute, it is preferable to have tractable stability conditions. Such conditions are conjectured in \cite{MiyaMoro2022a}, but they are disproved by \cite{KimKim2023}.

Both papers of those papers \cite{KimKim2023,MiyaMoro2022a} are based on Theorem 3.1 of the preprint \cite{MiyaMoro2022}. This theorem is correct, but some arguments in \cite{MiyaMoro2022a} are not accurate. These motivate us to write the present paper updating \cite{MiyaMoro2022a} taking \citet{KimKim2023} into account.

Thus, main contributions of this paper are a proof technique for the stability problem and the full characterization of the stability for $k=2$. The proof technique is a combination of sample path analysis similar to fluid approximation and the characterization of the stability based on the so-called Harris irreducibility of a Markov $T$-process with a general state space (see \pro{positive 1}). This approach is different from the method of \citet{MoroStey2013}, which is based on a regenerative process, so is restrictive in applications. Another typical approach to stability problems is to use fluid approximation and Lyapunov functions (see, e.g. \cite{Dai1995}). The fluid approximation largely depends on strong law of large numbers, strong LLN for short. However, it may have certain limitations. For example, \citet{Tezc2013} shows insufficiency for studying stability through fluid approximation (see also \citet{CherFossKim2013}). We encounter a similar situation because the stability condition for $k \ge 3$ requires $\rho^{*}_{j}(\xi)$ for downward stations, which can not be obtained from the fluid approximation. The counterexamples of \cite{KimKim2023} supports this situation.

This paper is made up by six sections. In \sectn{Markov}, we introduce the well-known framework for a Markov process to be positive Harris recurrent, which provides a base for our analysis. In \sectn{two station}, the 2-station cascade system is detailed, and its stability conditions are presented in \thr{stability 1}, which are proved in \sectn{system}. The general $k$-station cascade system is considered in \sectn{general}. Finally, \sectn{concluding} remarks the possibility to relax the renewal assumptions on the 2-station cascade system concerning its stability. The proofs for some auxiliary results are given in the appendix.

\section{Markov process and Harris recurrence}
\label{sec:Markov}
\setnewcounter

Our main concern is the stability of the $2$-station cascade system, which will be described by a Markov process. For formally discussing it, we use the following basic notations. Let $(\Omega,\sr{F},\dd{P})$ be a probability space on which a continuous-time stochastic process $X(\cdot) \equiv \{X(t); t \ge 0\}$ with state space $S$ is defined, where $S$ is a separable and locally compact metric state space. We assume that its sample path is right-continuous and has left-limits, and $X(\cdot)$ is adapted to a filtration $\dd{F} \equiv \{\sr{F}_{t}; t \ge 0\}$, that is, $X(t)$ is $\sr{F}_{t}$-measurable for all $t \ge 0$. Assume that $X(\cdot)$ is a strong Markov process with respect to $\dd{F}$.

For this Markov process $X(\cdot)$, we briefly introduce Harris irreducibility and (null and positive) recurrence. Let  $\sr{B}(S)$ be the Borel field on $S$, that is, the $\sigma$-field on $S$ generated by all open sets of $S$, and let $\dd{P}_{x}(A) = \dd{P}(A|X(0)=x)$ for $x \in S$ and $A \in \sr{F}$. The following definitions are taken from \citet{MeynTwee1993a} (for their discrete counterparts, see Section 8.3 of \citet{MeynTwee2009}).  In what follows, a measure $\varphi$ on $S$ means that it is defined on measurable space $(S,\sr{B}(S))$, and it is said to be non-trivial if $\varphi(S)>0$.

\begin{definition}[Irreducibility and recurrence]
\label{dfn:Harris 1}
For the Markov process $X(\cdot)$, let
\begin{align*}
  \tau_{A} = \inf \{t \ge 0; X(t) \in A\}, \qquad \eta_{A} = \int_{0}^{\infty} 1(X(u) \in A) du, \qquad A \in \sr{B}(S),
\end{align*}
where $1(\cdot)$ denotes the indicator function of proposition ``$\cdot$''.
Then, $X(\cdot)$ is called Harris irreducible, Harris recurrent, positive Harris recurrent if the following conditions (a), (b) and (c) are satisfied, respectively.
\begin{itemize}
\item [(a)] There is a non-trivial $\sigma$-finite measure $\varphi$ on $S$, called an irreducibility measure, such that, for $\forall B \in \sr{B}(S)$, $\varphi(B) > 0$ implies that $\dd{E}_{x}(\eta_{B}) > 0$, for $\forall x \in S$. In this case, $X(\cdot)$ is called $\varphi$-irreducible if $\varphi$ is specified.

\item [(b)] There is a non-trivial $\sigma$-finite measure $\varphi$ on $S$ such that $\varphi(B) > 0$ implies that $\dd{P}_{x}(\eta_{B} = \infty)=1$ for $\forall x \in S$, which is equivalent to that there is a non-trivial $\sigma$-finite measure $\varphi$ on $S$ such that $\varphi(B) > 0$ implies that $\dd{P}_{x}(\tau_{B} < \infty)=1$ for $\forall x \in S$. In this case, there exists a unique invariant measure on $S$ up to a multiplication constant, where $\sigma$-finite measure $\nu$ on $S$ is called an invariant measure if
\begin{align*}
  \nu(A) = \int_{S} \nu(dx) \dd{P}_{x}(X(t) \in A), \qquad \forall A \in \sr{B}(S), \forall t > 0.
\end{align*}

\item [(c)] $X(\cdot)$ is Harris recurrent, and there is an invariant probability measure on $S$, that is, $\nu$ of (b) is a probability measure.
\end{itemize}
In particular, we call $X(\cdot)$ and the stochastic model described by it to be stable if $X(\cdot)$ is positive Harris recurrent.
\end{definition}

\begin{definition}[$K$-chain and $T$-process]
\label{dfn:T-process 1}
(i) For a probability measure $a$ on $(\dd{R}_{+},\sr{B}(\dd{R}_{+}))$, define kernel $K_{a}$ as
\begin{align*}
  K_{a}(x,A) = \int_{0}^{\infty} \dd{P}_{x}(X(t) \in A) a(dt), \qquad x \in S, A \in \sr{B}(S),
\end{align*}
then $K_{a}$ is called a $K$-chain with sampling distribution $a$.\\
(ii) 
The Markov process $X(\cdot)$ is called a $T$-process if there is a kernel $T(x,A)$ for $x \in S$ and $A \in \sr{B}(S)$ such that $T$ is non-trivial, that is, $T(x,S) > 0$ for $\forall x \in S$, $T(x,A)$ is lower semi-continuous in $x \in S$, that is, $T(x,A) \le \liminf_{y \to x} T(y,A)$, and the following inequality holds for some probability measure $a$ on $\dd{R}_{+}$.
\begin{align*}
  K_{a}(x,A) \ge T(x,A), \qquad \forall x \in S, \forall A \in \sr{B}(S).
\end{align*}
\end{definition}

We first note some basic facts for a Harris recurrent Markov process. They are not only used in our proofs, but also helpful to better understand our arguments.

\begin{lemma}\rm
\label{lem:compact 1}[The continuous-time counterpart of Proposition 3.4 of \cite{TuomTwee1979}]
Assume that the Markov process $X(\cdot)$ is a Harris recurrent $T$-process. Denote its $\sigma$-finite invariant measure by $\nu$. Then, $\nu(C)$ is finite for each compact set $C$ of $S$.
\end{lemma}

This lemma can be proved in the exactly same way as Proposition 3.4 in \cite{TuomTwee1979} for a discrete-time Markov process. This is because the $K$-chains play the exactly same role in both cases. The essence here is that $\{x \in S; T(x,A) > 0\}$ is an open set of $S$ because $T(x,A)$ is lower semi-continuous in $x \in S$ for each $A \in \sr{B}(S)$.

The next fact is known as the ratio limit (or ergodic) theorem.

\begin{proposition}[Remark 1 for Theorems II.1 and II.2 of \cite{AzemKaplRevu1967}]                                                                                                                                                                                                                                                                                                                                                                                                                                                                                                                                                                                                                                                                                                                                                                                                                                                                                                                                                                                                                                                                                                                                                                                                                                                                                                                                                                                                                                                                                                                                                                                                                                                       
\label{pro:ergodic 1}
Assume that the Markov process $X(\cdot)$ is Harris recurrent, and denote its $\sigma$-finite invariant measure on $S$ by $\nu$. Let $L^{1}_{+}(\nu)$ be the set of all nonnegative measurable functions $f$ on $S$ such that $\br{\nu,f} \equiv \int_{S} f(y) \nu(dy) < \infty$, then, for $f, g \in L^{1}_{+}(\nu)$ satisfying $\br{\nu,g} > 0$,
\begin{align}
\label{eq:ergodic 1}
  \lim_{t \to \infty} \frac {\int_{0}^{t} f(X(u)) du}{\int_{0}^{t} g(X(u)) du} = \frac {\br{\nu,f}} {\br{\nu,g}}, \qquad a.s. \; \dd{P}_{x}, \forall x \in S.
\end{align}
\end{proposition}

\begin{remark}\rm
\label{rem:ergodic 2}
(i) If $X(\cdot)$ is positive recurrent, then $\nu(S) < \infty$, so, for any compact set $C \subset S$, $\nu(C) \le \nu(S) < \infty$. Hence, normalizing $\nu$ as $\nu(S)=1$, the dominated convergence theorem and \eq{ergodic 1} with $f(x) = 1(x \in C)$ and $g(x) = 1$ for $x \in S$ yield
\begin{align}
\label{eq:ergodic-p 1}
  \lim_{t \to \infty} \frac {1}{t} \int_{0}^{t} \dd{P}_{x}(X(u) \in C) du = \dd{E}_{x}\left(\lim_{t \to \infty} \frac {1}{t} \int_{0}^{t} 1(X(u) \in C) du\right) = \nu(C).
\end{align}
(ii) If $X(\cdot)$ is null recurrent, then $\nu(S) = \infty$. Since $\nu$ is $\sigma$-finite, there is a sequence of sets $\{S_{i} \in \sr{B}(S); i \ge 1\}$ such that $\nu(S_{i}) < \infty$ for each $i \ge 1$, $S_{i} \uparrow S$ and $\nu(S_{i}) \to \infty$ as $i \to \infty$. Since, for any compact set $C \subset S$, $\nu(C) < \infty$ by \lem{compact 1}, we have, by \eq{ergodic 1}, under $\dd{P}_{x}$,
\begin{align*}
   \limsup_{t \to \infty} & \frac {1}{t} \int_{0}^{t} 1(X(u) \in C) du \le \lim_{t \to \infty} \frac {\int_{0}^{t} 1(X(u) \in C) du} {\int_{0}^{t} 1(X(u) \in S_{i})du} = \frac {\nu(C)} {\nu(S_{i})} \to 0,\mbox{ as } i \to \infty.
\end{align*}
This implies that $\lim_{t \to \infty} \frac {1}{t} \int_{0}^{t} 1(X(u) \in C) du = 0$ $a.s.$ $\dd{P}_{x}$. Hence, by the dominated convergence theorem, $\lim_{t \to \infty} \dd{E}_{x}\left(\frac {1}{t} \int_{0}^{t} 1(X(u) \in C) du\right) = 0$. Since this limit can be replaced by limit supremum,
\begin{align}
\label{eq:ergodic-n 1}   \limsup_{t \to \infty} & \frac {1}{t} \int_{0}^{t} \dd{P}_{x}(X(u) \in C) du = \limsup_{t \to \infty} \dd{E}_{x}\left(\frac {1}{t} \int_{0}^{t} 1(X(u) \in C) du\right) = 0.
\end{align}
\end{remark}

The following proposition is obtained in \cite{MeynTwee1993a}, which will be a key for our approach.

\begin{proposition}[(ii) of Theorem 3.2 and the second half of (iv) of Theorem 3.4 of \cite{MeynTwee1993a}]                                                                                                                                                                                                                                                                                                                                                                                                                                                                                                                                                                                                                                                                                                                                                                                                                                                                                                                                                                                                                                                                                                                                                                                                                                                                                                                                                                                                                                                                                                                                                                                                                                               
\label{pro:positive 1}
Assume that the Markov process $X(\cdot)$ is a $T$-process and Harris irreducible, then the following three conditions are equivalent.
\begin{itemize}
\item [(a)] $X(\cdot)$ is positive Harris recurrent.
\item [(b)] For any $x \in S$ and any $\varepsilon > 0$, there is a compact set $C \subset S$ such that
\begin{align}
\label{eq:tight 1}
  \liminf_{t \to \infty} \frac 1t \int_{0}^{t} \dd{P}_{x}(X(u) \in C) du \ge 1 - \varepsilon. 
\end{align}
\item [(c)] For some 
$ x \in S$ and some compact set  $C \subset S$, 
\begin{align}
\label{eq:positive 1}
  \limsup_{t \to \infty} \frac 1t \int_{0}^{t} \dd{P}_{x}(X(u) \in C) du > 0. 
\end{align}
\end{itemize}
\end{proposition}
\begin{remark}\rm
\label{rem:positive 1}
The condition (b) is called bounded in probability in \cite{MeynTwee1993a}, but we call it tight on average.
\end{remark}

This proposition is a subset of Theorems 3.2 and 3.4 of \cite{MeynTwee1993a}, and follows from \pro{ergodic 1}. However, the equivalence of (a) and (c) is not fully proved in \cite{MeynTwee1993a}. We prove this equivalence in \app{positive 1} because it is particularly important for this research.

\section{Two-station cascade system}
\label{sec:two station}
We focus on the two-station cascade system. This is a queueing system with single-server stations 1, 2  and infinite-buffers. The servers are working in parallel, and, for $i=1,2$, exogenous customers arrive at station $i$, and are called class-$i$ customers. Let $\{t_{e,i}(n); n \ge 1\}$ be the arriving times of class-$i$ customers, which are
assumed to be a renewal process with $i.i.d.$ interarrival times $\tau_{e,i}(n) = t_{e,i}(n) - t_{e,i}(n-1) > 0$, $n \ge 1$, where $t_{e,i}(0) = 0$. Denote the generic element of $\tau_{e,i}(n)$ by $\tau_{e,i}$, and the arrival rate by
\begin{align*}
  \lambda_{i} \equiv 1/\dd{E}(\tau_{e,i})\in (0,\,\infty),\qquad i = 1,2.
\end{align*}
This convention $\tau_{e,i}$ for $\tau_{e,i}(n)$ will be used for other random variables indexed by $n$. 
 Class-$i$ customers are served in the first-come-first-served manner, and their service times $\{\tau_{s,i}(n), n \ge 1\}$ are $i.i.d.$, and the service rate is denoted by $\mu_{i} \equiv 1/\dd{E}(\tau_{s,i}) > 0$ for $i = 1,2$. 

Let $Q_{1}(t)$ be the number of class-$1$ customers at station $1$ including customers being served at time $t$. There is a switching rule for a class-$1$ customer in such a way that, if a class-$1$ customer who is waiting for service or just arrives from the outside observes queue size including itself to be greater than $c_{1} \ge 1$ (a given threshold) and if station $2$ is empty, then it immediately goes to station $2$ as a class-${1|2}$ customer, and gets service immediately with $i.i.d.$ service times $\{\tau_{s,1|2}(n); n \ge 1\}$ with rate $\mu_{{1|2}} \equiv 1/\dd{E}(\tau_{s,{1|2}})$. We assume  that class-${1|2}$ customer is always preempted in service by class-$2$ customers and resumes its service when no class $2$ customer is in station 2. Thus, there is at most one class-${1|2}$ customer in station $2$ who can get service only when there is no class-$2$ customer. All customers who completed their service leave the system.

To consider the stability, we describe the 2-station cascade system by a Markov process. For $v =1,2,{1|2}$, let $Q_{v}(t)$ be the number of class-$v$ customers in the system at time $t$, which is called a queue $v$ at time $t$. Define $L_{i}(t)$ for $i=1,2$ as
\begin{align*}
  L_{1}(t) = Q_{1}(t), \qquad L_{2}(t) = Q_{2}(t) + Q_{{1|2}}(t), \qquad t \ge 0.
\end{align*}
Thus, $L_{i}(t)$ is the total number of customers in station $i$ at time $t$ for $i=1,2$. For time $t \ge 0$, let $R_{e,i}(t)$ and $R_{s,i}(t)$ be the remaining arrival and service times, respectively, of class-$i$ customers for $i=1,2$, and let $R_{s,{1|2}}(t)$ be the remaining service time of a class-${1|2}$ customer, where the remaining service times vanish if there is no customer of the corresponding class in the system. Let $S = \dd{Z}_{+}^{2} \times \{0,1\} \times \dd{R}_{+}^{5}$, which can be considered as a separable complete metric space, where the discrete topology is taken on $\dd{Z}_{+}^{2} \times \{0,1\}$. Define $\vc{X}(\cdot) = \{X(t); t \ge 0\}$ as
\begin{align}
\label{eq:X 1}
 & \vc{X}(t) = (\vc{Q}(t), \vc{R}_{e}(t), \vc{R}_{s}(t), R_{s,{1|2}}(t)) \in S,
\end{align}
where $\vc{Q}(t) = (Q_{1}(t),Q_{2}(t),Q_{1|2}(t))$, $\vc{R}_{e}(t) = (R_{e,1}(t),R_{e,2}(t))$ and $\vc{R}_{s}(t) = (R_{s,1}(t),R_{s,2}(t))$. We take the natural filtration $\dd{F} \equiv \{\sr{F}_{t}; t \ge 0\}$ for $\vc{X}(\cdot)$. Namely, $\sr{F}_{t}$ is defined as
\begin{align*}
  \sr{F}_{t} = \sigma(\{\vc{X}(u); u \in [0,t]\}), \qquad t \ge 0,
\end{align*}
where $\sigma(\cdot)$ is the minimal $\sigma$-field including all events in $\sr{F}$ generated by ``$\cdot$''. Then, clearly $\vc{X}(\cdot)$ is a continuous-time Markov process with respect to $\dd{F}$. Since its sample paths are piecewise deterministic, $\vc{X}(\cdot)$ is a strong Markov processes with respect to $\dd{F}$ (e.g., see \cite{Davi1993}). This process fully describes the 2-station cascade system, and we refer to it as a {\it 2-station cascade process}. Similarly, we define
\begin{align}
\label{eq:X2 1}
  \vc{X}_{2}(\cdot) = \{(Q_{2}(t),R_{e,2}(t), R_{s,2}(t); t \ge 0\},
\end{align}
which is a strong Markov process describing the queue of class-$2$ customers.

For the Markov process $\vc{X}(\cdot)$, we apply the framework introduced in \sectn{Markov}. The following assumption enables $\vc{X}(\cdot)$ to be a $T$-process and Harris irreducible.

\begin{assumption}
\label{ass:spread 1}
For $i=1,2$, (a) $\dd{P}(\tau_{e,i}>x) > 0$ for all $x > 0$, and (b) there is a function $p_{i}(x) \ge 0$ such that $\int_{0}^{\infty} p_{i}(x) dx > 0$ and, for some integer $j_{i} \ge 1$,
\begin{align*}
  \dd{P}\Big(a \le \sum_{n=1}^{j_{i} } \tau_{e,i}(n) \le b\Big) \ge \int_{a}^{b} p_{i}(x) dx, \qquad \forall a, b \mbox{ satisfying } \; 0 \le a < b.
\end{align*}
\end{assumption}

The following lemmas are easy consequences of this assumption.

\begin{lemma}\rm
\label{lem:Harris 1}
Under \ass{spread 1}, (i) the Markov process $\vc{X}(\cdot)$ is a $T$-process, and any $\vc{x}^{*} \in \{(0,0,0,\vc{r}_{e},0,0,0) \in S; \vc{r}_{e} \in \dd{R}_{+}^{2}\}$ is reachable from any $\vc{x} \in S$, that is, $\int_{0}^{\infty }\dd{P}_{\vc{x}}(\vc{X}(u) \in G ) du > 0$ for any $\vc{x} \in S$ and any open set $G$ containing $\vc{x}^{*}$,\\
(ii) $\vc{X}(\cdot)$ is $\varphi$-irreducible for $\varphi(B) \equiv T(\vc{x}^{*},B)$ for $B \in \sr{B}(S)$.
\end{lemma}

The part (i) of this lemma is similarly proved as Lemma 3.7 of \cite{MeynDown1994}, while (ii) is the continuous-time counterpart of Proposition 6.2.1 of \cite{MeynTwee2009}, and can be similarly proved. Nevertheless, we prove (ii) in \app{Harris-ii} for the completeness.

\begin{lemma}\rm
\label{lem:positive 2}
Under \ass{spread 1}, the Markov process $\vc{X}(\cdot)$ is positive Harris recurrent if and only if, for each $i=1,2$, there exist $\vc{x}_{i} \in S$ and $\ell_{i} \ge 0$ such that
\begin{align}
\label{eq:positive 2}
  \limsup_{t \to \infty} \frac 1t \int_{0}^{t} \dd{P}_{\vc{x}_{i}}(Q_{i}(u) \le \ell_{i}) du > 0, \qquad i=1,2.
\end{align}
\end{lemma}

This lemma is proved in \app{positive 2}. We are ready to answer to the stability of the 2-station cascade system.

\begin{theorem}\rm
\label{thr:stability 1}
Under \ass{spread 1}, the Markov process $\vc{X}(\cdot)$ 
is positive recurrent if and only if
\begin{align}
\label{eq:stability 1}
 & \widetilde{\rho}_{1} \equiv \frac {\lambda_{1}} {\mu_{1} + \mu_{{1|2}}(1-\rho_{2})} < 1,\\
\label{eq:stability 2}
 & \rho_{2} < 1.
\end{align}
\end{theorem}

\begin{remark} \rm
\label{rem:stability 1}
\citet[Theorem 3]{MoroStey2013} show that the essentially same condition is sufficient for the stability under the assumption that station $1$ has $m_{1}$ $( \ge 1)$ servers and Poisson arrivals and the switching level $c_{1} = m_{1}$. The extension from a single server to multiple servers at stations is not hard in our arguments (see \sectn{proof}). Thus, the sufficiency of \eq{stability 1} and \eq{stability 2} is essentially known under the assumption that $\vc{X}(\cdot)$ is a regenerative process, but their necessity is only considered in some cases in \cite{MoroStey2013}. For the general $k$-station cascade system at least for the Poisson arrivals, \citet{DelgMoro2014} study its stability, assuming $c_{i} = 1$. They use a fluid approximation, and their sufficient conditions for $k=2$ agree with ours. They also derive the necessary conditions, which are different from the sufficient ones. Namely, there is a gap in those conditions even for $k=2$.
\end{remark}

One may interpret $\widetilde{\rho}_{1}$ of \eq{stability 1} as a kind of the traffic intensity of exogenously arriving customers at station 1 provided station 1 is saturated (permanently busy). Hence, it is different from the actual traffic intensity which agrees with the stationary probability that station 1 is not empty when the system is stable. Thus, it may be interesting to compare $\widetilde{\rho}_{1}$ with $\rho^{*}_{1}(\xi)$, which is intuitively introduced as the effective traffic intensity in \sectn{introduction}. We now formally define it for a distribution $\xi$ on $S$ for $i=1,2$ by
\begin{align}
\label{eq:rho* 1}
  \rho^{*}_{i}(\xi) = \liminf_{t \to \infty} \frac 1t \int_{0}^{t} \Big( \int_{S} \dd{P}_{\vc{x}}(Q_{i}(u) \ge 1) \xi(d\vc{x}) \Big) du, \end{align}
which is called an effective traffic intensity of $Q_{i}(\cdot)$. We note the following basic facts on $\rho^{*}_{i}(\xi)$. They may look intuitively obvious, but have the important message that the stability problem on $\vc{X}(\cdot)$ can be answered by the effective traffic intensities.
\begin{lemma}\rm
\label{lem:stability 2}
Under \ass{spread 1}, (i) if $\vc{X}(\cdot)$ is positive recurrent, then, for any distribution $\xi$ on $S$,
\begin{align}
\label{eq:rho-1 1}
  \rho_{i}^{*}(\xi) = 1 - \dd{P}_{\nu}(Q_{i}(0)=0), \qquad i=1,2,
\end{align}
where $\nu$ is the stationary distribution of $\vc{X}(\cdot)$, and $\dd{P}_{\nu}(A) = \int_{S} \dd{P}_{\vc{x}}(A) \nu(dx)$ for $A \in \sr{F}$. Hence, $\rho^{*}_{i}(\xi)$ is independent of $\xi$ when $\vc{X}(\cdot)$ is positive recurrent, (ii) $\vc{X}(\cdot)$ is positive recurrent if and only if $\rho^{*}_{i}(\xi) < 1$ for $i=1,2$ and some distribution $\xi$ on $S$.
\end{lemma}

This lemma is proved in \app{stability 2}, independently of \thr{stability 1}, and will be used to prove the necessity of \eq{stability 1} in \thr{stability 1}. Thus, $\rho^{*}_{1}(\xi)$ is hard to compute, but still useful.

\section{System dynamics and proof of \thr{stability 1}}
\label{sec:system}
\setnewcounter

To prove \thr{stability 1}, we first detail the dynamics of the Markov process $\vc{X}(\cdot)$.

\subsection{Dynamics of the $2$-station cascade system}

Recall that $t_{e,i}(n)$ is the $n$-th arriving time of class-$i$ customer at station $i$ for $i=1,2$, and let $t_{e,1|2}(n)$ be the $n$-th arriving time of class-$1|2$ customer at station $2$ from station $1$. For $v=1,2,{1|2}$, define $A_{v}(\cdot) = \{A_{v}(t); t \ge 0\}$ as 
\begin{align*}
  A_{v}(t) = \sum_{n=1}^{\infty} 1(0 < t_{v}(n) \le t).
\end{align*}
That is, $A_{v}(\cdot)$ is the counting process of arrivals of class-$v$ customers.

We similarly introduce notations for departures. Let $\{t_{d,i}(n); n \ge 1\}$ be departure instants of class-$i$ customers to the outside for $i = 1,2$, and let $\{t_{d,1|2}(n); n \ge 1\}$ be departure instants of class-$1|2$ customers at station 2. We denote the corresponding counting processes by $D_{v}(\cdot) \equiv \{D_{v}(t); t \ge 0\}$ for $v=1,2,{1|2}$, that is 
\begin{align*}
  D_{v}(t) = \sum_{n=1}^{\infty} 1(0 < t_{v}(n) \le t).
\end{align*}
Note that $A_{v}(0) = D_{v}(0)=0$ by those definitions.
Since there is at most one class-${1|2}$ customer at station $2$, we have
\begin{align}
\label{eq:D{1|2} 1}
  D_{{1|2}}(t) \le A_{{1|2}}(t) \le D_{{1|2}}(t)+1, \qquad t \ge 0.
\end{align}
We also have the following flow balance equations.
\begin{align}
\label{eq:L1 1}
 & Q_{1}(t) = Q_{1}(0) + A_{1}(t) - (D_{1}(t) + A_{{1|2}}(t)),\\
\label{eq:L2 12}
 & Q_{1|2}(t) = Q_{1|2}(0) + A_{{1|2}}(t) - D_{{1|2}}(t),\\
\label{eq:L2 2}
 & Q_{2}(t) = Q_{2}(0) + A_{2}(t) - D_{2}(t).
\end{align}

To describe departures by service times, we introduce the following notations.
\begin{align*}
 & B_{i}(t) = \int_{0}^{t} 1(Q_{i}(u) > 0) du, \qquad t \ge 0, i=1,2,\\
 & B_{1|2}(t) = \int_{0}^{t} 1(Q_{1|2}(u) > 0, Q_{2}(u) = 0) du
\end{align*}
and let $I_{2}(t) = t - B_{2}(t)$. We also introduce
\begin{align*}
  J_{{1|2}}(t) = \int_{0}^{t} 1(Q_{1}(u) > c_{1}, Q_{2}(u) = 0) du, \qquad t \ge 0,
\end{align*}
where $c_{1}$ is the threshold for switching from station 1 to station 2. Obviously, 
\begin{align}
\label{eq:B12 1}
  J_{{1|2}}(t) \le B_{{1|2}}(t) \le I_{2}(t), \qquad t \ge 0.
\end{align}

Denote  by $N_{v}(t)$ the counting processes for potential numbers of class-$v$ customers  served  up to time $t$, namely,
\begin{align}
\label{eq:potential 1}
  N_{v}(t) = \sup\left\{ n \ge 0; \sum_{\ell=1}^{n} \tau_{s,v}(\ell) \le t\right\}, \qquad t \ge 0, v=1,2,{1|2},
\end{align}
where $\tau_{s,1|2}(0)= 0$ and $\tau_{s,1|2}(n) = t_{s,1|2}(n) - t_{s,1|2}(n-1)$ for $n \ge 1$. Since $D_{v}(t) = N_{v}(B_{v}(t))$ and
\begin{align*}
  \{Q_{1}(u) > c_{1}, Q_{2}(u) = 0\} = \{Q_{1}(u) > c_{1}, Q_{2}(u) = 0, Q_{1|2}(u) > 0\}, \qquad 
   u \ge 0,
\end{align*}
it follows from \eq{B12 1} that
\begin{align}
\label{eq:D12 2}
  N_{{1|2}}(J_{{1|2}}(t)) \le D_{{1|2}}(t) \le N_{{1|2}}(I_{2}(t)), \qquad t \ge 0.
\end{align}

\subsection{Auxiliary lemmas}
\label{sec:auxiliary}

Recall the definition \eq{X2 1} of $\vc{X}_{2}(\cdot)$. Under the \ass{spread 1}, $\vc{X}_{2}(\cdot)$ is a Harris irreducible Markov process with state space $\dd{Z}_{+} \times \dd{R}_{+}^{2}$, and describes the $GI/G/1$ queue with traffic intensity $\rho_{2} = \lambda_{2}/\mu_{2}$. As is well known, $\vc{X}_{2}(\cdot)$ is positive recurrent if and only if $\rho_{2} < 1$. Assume that $\rho_{2} < 1$, then, by Little's law for a customer in service,
\begin{align*}
  \dd{P}_{\nu_{2}}(Q_{2}(0) \ge 1) = \lambda_{2} / \mu_{2} = \rho_{2},
\end{align*}
where $\dd{P}_{\nu_{2}}(Q_{2}(0) \ge 1) = \int_{\dd{Z}_{+} \times \dd{R}_{+}^{2}} \dd{P}(Q_{2} \ge 1|\vc{X}_{2} = \vc{y}) \nu_{2} (d\vc{y})$ for the stationary distribution $\nu_{2}$ of $\vc{X}_{2}(\cdot)$, and therefore (i) of \rem{ergodic 2} yields
\begin{align}
\label{eq:I2 1}
  \lim_{t \to \infty} \frac 1t I_{2}(t) = \lim_{t \to \infty} \frac 1t \int_{0}^{t} 1(Q_{2}(u) = 0) du = 1 - \rho_{2}, \qquad w.p.1.
\end{align}

\begin{lemma}\rm
\label{lem:D12 3}
If $\rho_{2} = \lambda_{2}/\mu_{2} < 1$, then
\begin{align}
\label{eq:D12 3}
  \ol{D}_{{1|2}} \equiv \limsup_{t \to \infty} \frac 1t D_{{1|2}}(t) \le \mu_{{1|2}} (1 - \rho_{2}), \qquad w.p.1,
\end{align}
and therefore
\begin{align}
\label{eq:lambda12 1}
  \lambda_{1|2} \equiv \dd{E}\left(\ol{D}_{{1|2}}\right) \le \mu_{{1|2}} (1 - \rho_{2}).
\end{align}
\end{lemma}
\begin{proof}
Since 
$$
\sum_{\ell=1}^{N_{{1|2}}(t)} \tau_{s,1|2}(\ell) \le t < \sum_{\ell=1}^{N_{{1|2}}(t)+1} \tau_{s,1|2}(\ell),
$$ 
strong LLN yields
\begin{align}
\label{eq:N12 1}
  \lim_{t \to \infty} \frac 1t N_{{1|2}}(t) = \mu_{{1|2}}, \qquad w.p.1.
\end{align}
Combing this with \eq{D12 2} and \eq{I2 1}, we have
\begin{align*}
  \ol{D}_{{1|2}}(t) \le \limsup_{t \to \infty} \frac {I_{2}(t)} t \frac 1{I_{2}(t)} N_{{1|2}}(I_{2}(t)) = \mu_{{1|2}}(1-\rho_{2}), \qquad w.p.1,
\end{align*}
which proves \eq{D12 3}.
\end{proof} 

We will prove \thr{stability 1} by applying Lemmas \lemt{positive 2} and \lemt{stability 2}. For the sufficiency of the conditions \eq{stability 1} and \eq{stability 2}, we will use the following lemma.

\begin{lemma}\rm
\label{lem:D12 4}
Assume that $\rho_{2} = \lambda_{2}/\mu_{2} < 1$ and
\begin{align}
\label{eq:not-tight 1}
  \limsup_{t \to \infty} \frac 1t \int_{0}^{t} \dd{P}_{\vc{x}}(Q_{1}(u) \le \ell) du = 0,
\qquad \forall \vc{x} \in S, \forall \ell \ge 0.
\end{align}
Then, there exists a nonnegative sequence $\{s_{n}; n \ge 1\}$ such that $s_{n} \uparrow \infty$ and, under $\dd{P}_{x}$,
\begin{align}
\label{eq:D12 4}
  \liminf_{t \to \infty} \frac 1{s_{n}} D_{{1|2}}(s_{n}) \ge \mu_{{1|2}} (1 - \rho_{2}), \qquad w.p.1.
\end{align}
\end{lemma}
\begin{proof}
In what follow, we fix $\vc{X}(0) = \vc{x}$. Define $K_{1}(t)$ as
\begin{align*}
  K_{1}(t) = \int_{0}^{t} 1(Q_{1}(u) \le c_{1}) du, \qquad t \ge 0.
\end{align*}
Then, from the Markov inequality, for each $\varepsilon>0$, 
\begin{align*}
  \dd{P}(K_{1}(t) > t \varepsilon) \le \frac 1{\varepsilon} \dd{E}\Big(\frac 1t K_{1}(t)\Big),
\end{align*}
so \eq{not-tight 1} implies that
\begin{align*}
 \limsup_{t \to \infty}  \dd{P}\Big(\frac 1t K_{1}(t) > \varepsilon\Big) \le \frac 1{\varepsilon} \limsup_{t \to \infty} \dd{E}\Big(\frac 1t K_{1}(t)\Big) = \limsup_{t \to \infty} \frac 1{\varepsilon t} \int_{0}^{t} \dd{P}(Q_{1}(u) \le c_{1}) du = 0.
\end{align*}


Hence, $\frac 1t K(t)$ converges to $0$ in probability as $t \to \infty$. Then, as is well known (e.g., see Corollary 6.1.2 of \cite{Boro2013}), we can choose a subsequence $\{s_{n}; n \ge 1\}$ from $\dd{R}_{+}$ such that $s_{n} \to \infty$ as $n \to \infty$ and
\begin{align}
\label{eq:K1 1}
  \lim_{n \to \infty} \frac 1{s_{n}} K_{1}(s_{n}) = 0, \qquad w.p.1.
\end{align}
Since
\begin{align*}
  & J_{{1|2}}(t) = \int_{0}^{t} 1(Q_{2}(u) = 0) du - \int_{0}^{t} 1(Q_{1}(u) \le c_{1}, Q_{2}(u) = 0) du
\end{align*}
and $\int_{0}^{t} 1(Q_{1}(u) \le c_{1}, Q_{2}(u) = 0) du \le K_{1}(t)$, \eq{I2 1} and \eq{K1 1} yield
\begin{align*}
  \lim_{n \to \infty} \frac 1{s_{n}} J_{{1|2}}(s_{n}) = \lim_{n \to \infty} \frac 1{s_{n}} \int_{0}^{s_{n}} 1(Q_{2}(u) = 0) du = 1 - \rho_{2}, \qquad w.p.1,
\end{align*}
because $\vc{X}_{2}(\cdot)$ is positive Harris recurrent by $\rho_{2} < 1$. Hence, this together with \eq{D12 2} and \eq{N12 1} imply that
\begin{align*}
  \liminf_{n \to \infty} \frac 1{s_{n}} D_{{1|2}}(s_{n}) & \ge \liminf_{n \to \infty} \frac {J_{{1|2}}(s_{n})}{s_{n}} \frac 1{J_{{1|2}}(s_{n})} N_{{1|2}}(J_{{1|2}}(s_{n}))\\
  & = (1-\rho_{2}) \liminf_{t \to \infty} \frac 1{J_{{1|2}}(s_{n})} N_{{1|2}}(J_{{1|2}}(s_{n)})\\
  & = (1-\rho_{2}) \mu_{{1|2}}, \qquad w.p.1,
\end{align*}
where the last equality is obtained by \eq{N12 1}. Thus, we have proved \eq{D12 4}.
\end{proof} 

\subsection{The proof of \thr{stability 1}}
\label{sec:proof}

We now start to prove \thr{stability 1}. We first prove the sufficiency of \eq{stability 1} and \eq{stability 2}. We apply reduction to the absurdity. For this assume \eq{not-tight 1} for all $\vc{x} \in S$ and $\ell \ge 0$. Since $\rho_{2} < 1$, it follows from Lemmas \lemt{D12 3} and \lemt{D12 4} that, for the nonnegative sequence $\{s_{n}; n \ge 1\}$ obtained in \lem{D12 4},
\begin{align}
\label{eq:D12 5}
  \lim_{n \to \infty} \frac 1{s_{n}} A_{{1|2}}(s_{n}) = \lim_{n \to \infty} \frac 1{s_{n}} D_{{1|2}}(s_{n}) = \mu_{{1|2}} (1 - \rho_{2}), \qquad w.p.1.
\end{align}
Since $B_{1}(t) = t - \int_{0}^{t} 1(Q_{1}(u)=0) du$, we have, by \eq{K1 1},
\begin{align}
\label{eq:B1 1}
  \lim_{n \to \infty} \frac 1{s_{n}} B_{1}(s_{n})  = 1 - \lim_{n \to \infty} \frac 1{s_{n}} \int_{0}^{s_{n}} 1(Q_{1}(u)=0) du = 1, \qquad w.p.1.
\end{align}
Furthermore, because $D_{v}(t) = N_{v}(B_{v}(t))$, strong LLN and \eq{B1 1} yield
\begin{align}
\label{eq:A1 1}
 & \lim_{n \to \infty} \frac 1{s_{n}} A_{1}(s_{n}) = \lambda_{1}, \qquad w.p.1,\\
\label{eq:D1 1}
 &  \lim_{n \to \infty} \frac 1{s_{n}} D_{1}(s_{n}) = \lim_{n \to \infty} \frac 1{s_{n}} B_{1}(s_{n}) \frac 1{B_{1}(s_{n})} N_{1}(B_{1}(s_{n})) = \mu_{1}, \qquad w.p.1.
\end{align}
Hence, \eq{L1 1} yields 
\begin{align}
\label{eq:not-stable 1}
  \lambda_{1} - \mu_{1} - \mu_{{1|2}} (1- \rho_{2}) = \lim_{n \to \infty} \frac 1{s_{n}} Q_{1}(s_{n}) \ge 0, \qquad w.p.1,
\end{align}
which implies that
\begin{align}
\label{eq:not-stable 2}
  \lambda_{1} \ge \mu_{1} + \mu_{{1|2}} (1- \rho_{2}).
\end{align}
Consequently, \eq{not-tight 1} for all $\vc{x} \in S$ and $\ell > 0$ implies \eq{not-stable 2}. Thus, taking the contraposition of this implication, 
\begin{align}
\label{eq:stability 3}
  \lambda_{1} < \mu_{1} + \mu_{{1|2}} (1- \rho_{2})
\end{align}
implies that \eq{not-tight 1} does not hold, that is, for some $\vc{x} \in S$ and some $\ell \ge 0$,
\begin{align}
\label{eq:positive 3}
  \limsup_{t \to \infty} \frac 1t \int_{0}^{t} \dd{P}_{\vc{x}}(Q_{1}(u) \le \ell) du > 0.
\end{align}
Hence, $\vc{X}(\cdot)$ is positive Harris recurrent by \lem{positive 2}. 

For the necessity of \eq{stability 1} and \eq{stability 2}, assume that $\vc{X}(\cdot)$ is positive Harris recurrent. As noted at the beginning of \sectn{auxiliary}, this implies that $\rho_{2} < 1$. So, we only need to show the necessity of \eq{stability 1}. From the positive recurrence of $\vc{X}(\cdot)$, we have $\rho^{*}_{1}(\xi) = 1-\dd{P}_{\nu}(Q_{1}(0) = 0) < 1$ by \lem{stability 2}. By \lem{D12 3}, $\ol{D}_{{1|2}} \le \mu_{{1|2}} (1-\rho_{2})$. Hence, applying Little's law to the class-1 customers who are served at station 1 and their mean number in service, which is $\rho^{*}_{1}(\xi)$, we have
\begin{align}
\label{eq:stability 4}
  \mu_{1} \rho^{*}_{1}(\xi) = \lambda_{1} - \lambda_{{1|2}} \ge \lambda_{1} - \mu_{{1|2}}(1-\rho_{2}).
\end{align}
On the other hand, from the definition \eq{stability 1} of 
$ \widetilde{\rho}_{1}$,
$$
\mu_{{1|2}}(1-\rho_{2}) \widetilde{\rho}_{1} = \lambda_{1} - \mu_{1} \widetilde{\rho}_{1}.
$$ 
Hence, it follows from \eq{stability 4} that
\begin{align*}
  \mu_{1} \rho^{*}_{1}(\xi) \widetilde{\rho}_{1} \ge \lambda_{1} \widetilde{\rho}_{1} - \mu_{{1|2}}(1-\rho_{2}) \widetilde{\rho}_{1} = \lambda_{1} \widetilde{\rho}_{1} - \lambda_{1} + \mu_{1} \widetilde{\rho}_{1},
\end{align*}
which is equivalent to $\lambda_{1} (1 - \widetilde{\rho}_{1}) \ge \mu_{1} \widetilde{\rho}_{1} (1-\rho^{*}_{1}(\xi))$. Thus, $\rho^{*}_{1}(\xi) < 1$ implies $\widetilde{\rho}_{1} < 1$. This proves that \eq{stability 1} is necessary for the stability, and the proof of \thr{stability 1} is completed.

\section{General $k$-station cascade system}
\label{sec:general}
\setnewcounter

We next consider an extension of \thr{stability 1} to a cascade system with $k$-stations, which are numbered as $1,2,\ldots,k$ from top to bottom. Each station $i$ has exogenous arrivals, called class-$i$ customers, and served in the manner of first come first served. If the number of class-$i$ customers is greater than $c_{i}$ and if station $i+1$ is empty, then one class-$i$ customer not in service immediately moves to station $i+1$ as a class-$i|(i+1)$ customer at station $i+1$ for $i \le k-1$, where $c_{i}$ is a positive integer. It is assumed that a class-$i|(i+1)$ customer is only served when there is no class-$(i+1)$ customer, and its service is resumed when class-$(i+1)$ customers vanish at station $i+1$. All customers leave the system when their service is completed. We refer to this system as a $k$-station cascade system.

Note that the system is stochastically independent of which class-$i$ customer not in service moves to station $i+1$ as a class-$i|(i+1)$ customer. Furthermore, an exogenously arriving class-$(i+1)$ customer preempts class-$i|(i+1)$ customer when the latter is being served at station $(i+1)$. Hence, the stability of station $i$ is independent of the dynamics of stations $1,2,\ldots, i-1$, but it is influenced from downward stations $i+1, i+2, \ldots,k$. Let $Q_{i}(t)$ be the number of class-$i$ customers in station $i$ at time $t \ge 0$ and let $Q_{i-1|i}(t)$ the number of class-$(i-1)|i$ customers in station $i$ at time $t$. The remaining times $R_{e,i}(t)$, $R_{s,i}(t)$ for $i=1,2,\ldots,k$ and $R_{s,i|(i+1)}(t)$ for $i=1,2,\ldots,k-1$ are defined for each time $t \ge 0$ similarly to the case of $k=2$.

As for the regularity condition, we need to replace \ass{spread 1} by its $k$-station version. Namely,
\begin{assumption}
\label{ass:spread 2}
For $i=1,2,\ldots,k$, the conditions (a) and (b) of \ass{spread 1} hold.
\end{assumption}

Similarly to the case of $k=2$, we can construct the Markov process $\vc{X}(\cdot) \equiv \{\vc{X}(t); t \ge 0\}$ by using the remaining arrival and service times, where
\begin{align}
\label{eq:X k}
 & \vc{X}(t) = (\vc{Q}_{1,k}(t), \vc{R}_{e,1,k}(t), \vc{R}_{s,1,k}(t), \vc{R}_{s,1,k}^{+}(t)) \in S \equiv \dd{Z}_{+}^{2k} \times \{0,1\}^{k-1} \times \dd{R}_{+}^{3k -1},
\end{align}
where, for $i=1,2,\ldots,k$,
\begin{align*}
& \vc{Q}_{i,k}(t) = (Q_{i}(t), Q_{i+1}(t), \ldots,Q_{k}(t), Q_{i|(i+1)}(t),Q_{(i+1)|(i+2)}(t),\ldots, Q_{(k-1)|k}(t)),\\
& \vc{R}_{e,i.k}(t) = (R_{e,i}(t), R_{e,i+1}(t), \ldots, R_{e,k}(t)),\quad
 \vc{R}_{s,i,k}(t) =(R_{s,i}(t), R_{s,i+1}(t), \ldots, R_{s,k}(t)),\\
& \vc{R}_{s,i,k}^{+}(t) = (R_{s,i|(i+1)}(t), R_{s,(i+1)|(i+2)}(t), \ldots, R_{s,(k-1)|k}(t)).
\end{align*}

To describe the system dynamics, we also introduce counting processes. For $v = i$ and $v = (i-1)|i$, let $A_{v}(t)$ be the number of the arrivals of class-$v$ customers up to time $t \ge 0$, and let $N_{v}(t)$ be the number of potential service completions of class-$v$ customers by time $t$ for $v=1,2,\ldots,k$ and $v=1|2, 2|3, \ldots, (k-1)|k$. It is assumed that $A_{i}(\cdot)$ and $N_{i}(\cdot)$ for $i =1,2,\ldots,k$ and $N_{(i-1)|i}(\cdot)$ for $i=2,3,\ldots,k$ are independent renewal processes. Their rates are denoted by $\lambda_{i}$, $\mu_{i}$ and $\mu_{(i-1)|i}$, respectively.

Let $D_{v}(t)$ be the number of departures of class-$v$ customers by time $t \ge 0$. Let
\begin{align*}
  L_{i}(t) = Q_{i}(t) + Q_{(i-1)|i}(t)1(i \ge 2), \qquad i=1,2,\ldots, k.
\end{align*}
Note that class-$i$ and class-$(i-1)|i$ customers have different service time distributions for $i \ge 2$. From the modeling assumptions, we have, for $i=1,2,\ldots,k$,
\begin{align}
\label{eq:L i}
 & Q_{i}(t) = Q_{i}(0) + A_{i}(t) - (D_{i}(t) + A_{i|i+1}(t)), \quad t \ge 0,\\
 & Q_{(i-1)|i}(t) = Q_{(i-1)|i}(0) + A_{(i-1)|i}(t) - D_{(i-1)|i}(t), \quad t \ge 0,
\end{align}
where $A_{0|1}(t) = A_{k|(k+1)}(t) = 0$. Note that the term $A_{(i-1)|i}(t) - D_{(i-1)|i}(t)$ has no influence on the stability of $Q_{i}(t)$ because the assumption on class-$(i-1)|i$ customers implies
\begin{align}
\label{eq:}
  0 \le A_{(i-1)|i}(t) - D_{(i-1)|i}(t) \le 1, \qquad t \ge 0.
\end{align}

Let $\rho_{i} = \frac {\lambda_{i}}{\mu_{i}}$ for $i=1,2,\ldots,k$. Similar to \eq{rho* 1}, we define $\rho^{*}_{i}(\xi)$ for a distribution $\xi$ as
\begin{align}
\label{eq:rho* i1}
  \rho^{*}_{i}(\xi) \equiv \liminf_{t \to \infty} \frac 1t \int_{0}^{t} \Big( \int_{S} \dd{P}_{\vc{x}}(Q_{i}(u) \ge 1) \xi(d\vc{x}) \Big) du,  \qquad i=2,3,\ldots,k.
\end{align}
We also define a Markov process describing the last $k-i+1$ stations by
\begin{align*}
  \vc{X}_{i,k}(\cdot) \equiv \{(\vc{Q}_{i,k}(t), \vc{R}_{e,i,k}(t), \vc{R}_{s,i,k}(t), \vc{R}^{+}_{s,i,k}(t)); t \ge 0\}, \qquad i=1,2,\ldots,k.
\end{align*}
Then, the following lemma is proved in the exactly same way as Lemmas \lemt{stability 2} and \lemt{D12 3}.
\begin{lemma}\rm
\label{lem:Di 1}
Under \ass{spread 2}, for each $i = 1,2, \ldots,k$, $\vc{X}_{i,k}(\cdot)$ is positive Harris recurrent if and only if $\rho^{*}_{j}(\xi) < 1$ for $j=i, i+1, \ldots, k$ for some distribution $\xi$ on $S$. Furthermore, particularly when this is the case: for $i \ge 2$,
\begin{align}
\label{eq:Di 1}
 & \ol{D}_{{(i-1)|i}} \equiv \limsup_{t \to \infty} \frac 1t D_{{(i-1)|i}}(t) \le \mu_{{(i-1)|i}} (1 - \rho^{*}_{i}(\xi)), \qquad w.p.1,\\
\label{eq:Di 2}
 & \lambda_{(i-1)|i} \equiv \dd{E}\left(\ol{D}_{{(i-1)|i}}\right) \le \mu_{{(i-1)|i}} (1 - \rho^{*}_{i}(\xi)).
\end{align}
\end{lemma}

Then, we have the following result similar to \thr{stability 1}.

\begin{theorem}\rm
\label{thr:stability-k 1}
Under \ass{spread 2}, the Markov process $\vc{X}(\cdot)$ which describes the $k$-station cascade system is positive recurrent if and only if the following conditions hold for some distribution $\xi$ on $S$.
\begin{align}
\label{eq:stability-k 1}
 & \widetilde{\rho}_{i}(\xi) \equiv \frac {\lambda_{i}} {\mu_{i} + \mu_{{i|(i+1)}}(1-\rho^{*}_{i+1}(\xi))} < 1, \qquad i=1,2,\ldots, k-1,\\
\label{eq:stability-k 2}
 & \widetilde{\rho}_{k}(\xi) \equiv \rho_{k} < 1.
\end{align}
\end{theorem}

This theorem is proved in the exactly same way as \thr{stability 1} using the induction on $i$ from $i=k-1$ to $1$. We outline this proof below. 

The necessity of \eq{stability-k 1} and \eq{stability-k 2} can be proved in the same way as the proof of \thr{stability 1} using \lem{Di 1} instead of \lem{stability 2}. To see their sufficiency, suppose that $\widetilde{\rho}_{i}(\xi) < 1$ and $\rho^{*}_{j}(\xi) < 1$ for $j=i+1,i+2,\ldots,k$, then $\vc{X}_{i,k}(\cdot)$ is positive recurrent by a similar proof in \thr{stability 1} and \lem{Di 1}. Hence, $\rho^{*}_{i}(\xi) < 1$ by \lem{Di 1}. Inductively repeating this argument from $i=k-1$ to $i=1$, we can conclude that $\vc{X}_{1,k}(\cdot) (= \vc{X}(\cdot))$ is positive recurrent.

\begin{remark}\rm
\label{rem:stability-k 1}
Theorem of \cite{KimKim2023} formulates the stability problem in a slightly different way that the stability condition is checked backward induction from $k$ to $1$. Then, Theorem 1 of \cite{KimKim2023} says that $\vc{X}_{i,k}(\cdot)$ is stable if and only if $\widetilde{\rho}_{i}(\xi) < 1$ and $\vc{X}_{i+1,k}(\cdot)$ is stable. Obviously, this is essentially rephrasing \thr{stability 1} of \cite{MiyaMoro2022} (also of this paper). \thr{stability-k 1} is a little more than it due to \lem{Di 1}.
\end{remark}

\section{Concluding remarks}
\label{sec:concluding}
\setnewcounter

We have assumed that the exogenous arrival processes $A_{i}(\cdot)$ at station $i$ are independent renewal processes with finite arrival rates and the service times of each class of customers are $i.i.d.$. However, from the proof of \thr{stability 1}, those assumptions can be relaxed in proving that \eq{stability 1} implies \eq{positive 3} under the assumption $\rho_{2} < 1$, which is the core part in the proof of \thr{stability 1}. Namely, this implication is valid as long as the following conditions hold.
\begin{align}
\label{eq:general 1}
  & \lim_{t \to \infty} \frac 1t A_{i}(t) = \lambda_{i}, \qquad i=1,2,\; w.p.1,\\
\label{eq:general 2}
  & \lim_{n \to \infty} \frac 1n \sum_{\ell=1}^{n} \tau_{s,v}(\ell) = \mu_{v}, \qquad v=1,2,{1|2},\; w.p.1,\\
\label{eq:general 3}
  & \lim_{t \to \infty} \frac 1t \int_{0}^{t} 1(Q_{2}(s)=0) ds = \max(0,1 - \rho_{2}), \qquad w.p.1,
\end{align}
where $\rho_{2} = \lambda_{2}/\mu_{2}$.

The condition \eq{positive 3} is much weaker than the existence of the stationary distribution of $Q_{1}(\cdot)$, but it may be interpreted as a certain stability property of the process $Q_{1}(\cdot)$. Hence, it may be interesting to investigate the stability of the 2-station cascade system in a weaker sense for more general exogenous arrival processes based on the conditions \eq{general 1}, \eq{general 2} and \eq{general 3}.

\section*{\large Acknowledgements}

After the submission of this paper, we have learned that \citet{KimKim2023} disproves the conjecture of \cite{MiyaMoro2022a}. We thanks the Editor-in-chief for allowing us to revise the original submission taking \cite{KimKim2023} into account. We are grateful to the anonymous referee for pointing out various errors in the revised submissions and helpful suggestions for correcting them. 

\section*{Appendix}
\label{app:proofs}
\setnewcounter
\renewcommand{\thesubsection}{A.\arabic{subsection}}
\renewcommand{\theequation}{A.\arabic{equation}}
\setcounter{section}{1}

\subsection{The equivalence of (a) and (c) in \pro{positive 1}}
\label{app:positive 1}

We prove that (a) is equivalent to (c) in \pro{positive 1}. We first note that, under the assumption that $X(\cdot)$ is a Harris irreducible $T$-process, $X(\cdot)$ is either Harris recurrent or $\sigma$-transient by the irreducibility and the Doeblin decomposition (see Theorem 3.1 of \cite{TuomTwee1979a}), where $X(\cdot)$ is said to be $\sigma$-transient if $S$ is a countable union of measurable sets $A_{i}$ for $i \ge 1$ such that $\dd{E}_{x}(\eta_{A_{i}}) < \infty$ for all $x \in S$ and for all $i \ge 1$, in which case each $A_{i}$ is said to be uniformly transient. If $X(\cdot)$ is positive Harris recurrent, then it has an invariant probability measure. Denote it by $\nu$. Then, \eq{positive 1} is immediate from \eq{ergodic-p 1} in \rem{ergodic 2} because there exists a compact set $C \subset S$ such that $\nu(C) > 0$ by the tightness of $\nu$ on $S$.\\
Conversely, if \eq{positive 1} holds, then $\dd{E}_{x}(\eta_{C}) = \infty$ for the compact set $C$. Hence, $X(\cdot)$ can not be $\sigma$-transient, and therefore it is Harris recurrent by the Doeblin decomposition. Thus, there is a unique invariant measure. Suppose that this $X(\cdot)$ is null recurrent, then it follows from \eq{ergodic-n 1} in \rem{ergodic 2} that,
for all $x \in S$  and any compact set  $C \subset S$,
\begin{align}
\label{eq:null 1}
  \limsup_{t \to \infty} \frac 1t \int_{0}^{t} \dd{P}_{x}(X(u) \in C) du = 0.
\end{align}
This contradicts \eq{positive 1}, so $X(\cdot)$ must be positive Harris recurrent.

\subsection{Proof of (ii) of \lem{Harris 1}}
\label{app:Harris-ii}

By (i) of \lem{Harris 1} $\vc{X}(\cdot)$ is a $T$-process. Let $\varphi(B) = T(\vc{x}^{*},B)$, then $T(\vc{x}^{*},B) > 0$ implies that $T(\vc{y},B) > 0$ for $\forall \vc{y}$ in some open set $V(\vc{x}^{*})$ containing $\vc{x}^{*}$ by the lower semi-continuity of $T(\cdot,B)$. Furthermore, let $e$ be the exponential distribution with unit mean, then $K_{e}(\vc{x},V(\vc{x}^{*})) = \int_{0}^{\infty} \dd{P}_{\vc{x}}(\vc{X}(u) \in V(\vc{x}^{*})) e^{-u} du > 0$ for $\forall \vc{x} \in S$ since $\vc{x}^{*}$ is reachable from any $\vc{x} \in S$. Hence, for the sampling distribution $a$ by which $\vc{X}(\cdot)$ is a $T$-process, 
\begin{align*}
  K_{e * a}(\vc{x},B) \ge \int_{V_{\varepsilon}(\vc{x}^{*})} K_{e}(\vc{x},d\vc{y}) K_{a}(\vc{y}, B) \ge \int_{V_{\varepsilon}(\vc{x}^{*})} K_{e}(\vc{x},d\vc{y}) T(\vc{y}, B) > 0, \quad \forall \vc{x} \in S,
\end{align*}
where $e * a$ is the convolution of distributions $e$ and $a$ on $\dd{R}_{+}$, and $K_{e * a}(\cdot,\cdot)$ is $e * a$-sampling $K$-chain. This proves that $\vc{X}(\cdot)$ is $\varphi$-irreducible because distribution $e * a$ has an absolutely continuous component with respect to Lebesgue measure on $\dd{R}_{+}$.

\subsection{Proof of \lem{positive 2}}
\label{app:positive 2}

Obviously, \eq{positive 1} implies \eq{positive 2} for $X(\cdot) = \vc{X}(\cdot)$, so we only need to prove that \eq{positive 2} implies \eq{positive 1}. Recall that $\vc{X}_{2}(\cdot)$ is the Markov process describing the queue of class-2 customers. We first show that $\vc{X}_{2}(\cdot)$ is a positive Harris recurrent if \eq{positive 2} holds for $i=2$. Since $R_{e,2}(\cdot)$ is a regenerative process with cycles $\tau_{e,2}(\cdot)$ and $R_{s,2}(\cdot)$ is dominated by such a regenerative process $Y(\cdot)$ with cycles $\tau_{s,2}(\cdot)$ in the sense that
\begin{align*}
  \int_{0}^{t} 1(R_{s,2}(u) \le x) du \le \int_{0}^{t} 1(Y(u) \le x)du, \qquad x \ge 0, t \ge 0.
\end{align*}
both of $R_{e,2}(\cdot)$ and $R_{s,2}(\cdot)$ are tight on average. Hence, by \pro{positive 1}, for any $\vc{x}_{2} \in S$, any $\varepsilon > 0$ and $v = e, s$, there is a compact subset $C_{v,2} \subset \dd{R}_{+}$ such that
\begin{align}
\label{eq:tight av 1}
  \liminf_{t \to \infty} \frac 1t \int_{0}^{t} \dd{P}_{\vc{x}_{2}}(R_{v,2}(u) \in C_{v,2}) du \ge 1 - \varepsilon.
\end{align}
Let $\widetilde{C}_{2} = C_{2,e} \times C_{2,s}$, then
\begin{align*}
 \dd{P}_{\vc{x}_{2}}(\vc{X}_{2}(u) \in \{0\} \times \widetilde{C}_{2}) & \ge \dd{P}_{\vc{x}_{2}}(Q_{2}(u)=0) - \dd{P}_{\vc{x}_{2}}(R_{e,2}(u),R_{s,2}(u)) \not\in \widetilde{C}_{2})\\
  &  \ge \dd{P}_{\vc{x}_{2}}(Q_{1}(u)=0) - \big(\dd{P}_{\vc{x}_{2}}(R_{e,2}(u) \not\in C_{e,2}) + \dd{P}_{\vc{x}}(R_{s,2}(u) \not\in C_{s,2})\big)\\
  & \ge \dd{P}_{\vc{x}_{2}}(Q_{1}(u)=0) - 2 \varepsilon.
\end{align*}
Integrating this inequality for $u \in [0,t]$, it follows from \eq{positive 2} for $i=2$ that there is some $\vc{x}'_{2} \in S$,
\begin{align*}
    \limsup_{t \to \infty} \frac 1t \int_{0}^{t} \dd{P}_{\vc{x}'_{2}}(\vc{X}_{2}(u) \in \{0\} \times \widetilde{C}_{2}) du > 0.
\end{align*}
By \pro{positive 1}, this proves that $\vc{X}_{2}(\cdot)$ is positive recurrent. Furthermore, it is tight on average by (b) of \pro{positive 1}. Since $Q_{1|2}(t) \le 1$ for $\forall t \ge 0$, $\{Q_{1|2}(t); t \ge 0\}$ is obviously tight on average. Similar to $R_{e,2}(\cdot)$ and $R_{s,2}(\cdot)$, $R_{e,1}(\cdot)$, $R_{s,1}(\cdot)$ and $R_{s,1|2}(\cdot)$ are also tight on average. Hence, by a similar argument to $\vc{X}_{2}(\cdot)$, \eq{positive 2} for $i=1$ implies \eq{positive 1} for $X(\cdot) = \vc{X}(\cdot)$.

\subsection{Proof of \lem{stability 2}}
\label{app:stability 2}

Note that, under \ass{spread 1},  $\vc{X}(\cdot)$ is $\varphi$-irreducible $T$-process by \lem{Harris 1}. (i) Since $\vc{X}(\cdot)$ is positive Harris recurrent, it has the stationary distribution $\nu$, and, by \eq{ergodic-p 1} in \rem{ergodic 2},
\begin{align*}
  \lim_{t \to \infty} \frac 1t \int_{0}^{t} \dd{P}_{\vc{x}}(Q_{i}(u) \ge 1)du = \dd{P}_{\nu}(Q_{i} \ge 1), \qquad \forall \vc{x} \in S, i =1,2.
\end{align*}
Integrating both sides of this equation by $\xi$ and applying the dominated convergence theorem, we have
\begin{align*}
  \lim_{t \to \infty} \frac 1t \int_{0}^{t} \Big(\int_{S} \dd{P}_{\vc{x}}(Q_{i}(u) \ge 1) \xi(d\vc{x}) \Big) du = 1 - \dd{P}_{\nu}(Q_{i}(0)=0).
\end{align*}
This and the definition of $\rho^{*}_{i}(\xi)$ prove \eq{rho-1 1}.\\
(ii) By \lem{positive 2}, we only need to show that $\rho^{*}_{i}(\xi) < 1$ for $i=1,2$ is equivalent to \eq{positive 2} for some $\vc{x} \in S$ and some $\ell \ge 0$. Assume that $\rho^{*}_{i}(\xi) < 1$ for $i=1,2$. Because, by Fatou's lemma,
\begin{align*}
  \rho^{*}_{i}(\xi) \ge \int_{S} \Big(\liminf_{t \to \infty} \frac 1t \int_{0}^{t} \dd{P}_{\vc{x}}(Q_{i}(u) \ge 1) du \Big) \xi(d\vc{x}),
\end{align*}
we have
\begin{align*}
 & \int_{S} \Big( \limsup_{t \to \infty} \frac 1t \int_{0}^{t} \dd{P}_{\vc{x}}(Q_{i}(u) = 0) du \Big) \xi(d\vc{x}) \\
  & = \int_{S} \Big( 1- \liminf_{t \to \infty} \frac 1t \int_{0}^{t} \dd{P}_{\vc{x}}(Q_{i}(u) \ge 1) du \Big) \xi(d\vc{x}) \ge 1 - \rho^{*}_{i}(\xi) > 0.
\end{align*}
Hence, \eq{positive 2} holds for $i=1,2$, $\ell=0$ and some $\vc{x}_{i} \in S$.

Conversely, assume \eq{positive 2} for $i=1,2$ and some $\vc{x}_{i} \in S$ and some $\ell = 0$. Then, $\vc{X}(\cdot)$ is positive Harris recurrent by \lem{positive 2}. Hence, \eq{rho-1 1} holds by (i) of \lem{stability 2}, and therefore we only need to show that $\dd{P}_{\nu}(Q_{i}(0)=0) > 0$ for $i=1,2$. To prove this, we first note that $\vc{X}(\cdot)$ is Harris $\varphi$-irreducible for $\varphi(B) = T(\vc{x}^{*},B)$ by \lem{Harris 1}. Hence, $\varphi(B) > 0$ implies $K_{a}(\vc{x},B) > 0$ for $\forall \vc{x} \in S$ for some sampling distribution $a$ by \dfn{T-process 1}. Then, for the stationary distribution $\nu$ of $\vc{X}(\cdot)$, for any sampling distribution $a$ on $\dd{R}_{+}$,
\begin{align*}
  \nu(B) \ge \int_{S} \nu(d\vc{x}) K_{a}(\vc{x},B) > 0.
\end{align*}
Hence, $\varphi(B) > 0$ implies $\nu(B) > 0$ for any $B \in \sr{B}(S)$. From the definition of $\varphi$, $\varphi(G) > 0$ for any open set $G$ containing $\vc{x}^{*}$ of \lem{Harris 1}. As this $G$, we choose $G_{0}$ defined by
\begin{align*}
  G_{0} = \{0\} \times \dd{Z}_{+} \times \{0,1\} \times \dd{R}_{+}^{5} \subset S,
\end{align*}
which obviously contains $\vc{x}^{*}$. Since the discrete topology is taken on $\dd{Z}_{+}^{2} \times \{0,1\}$, $G_{0}$ is an open set. Hence, $\varphi(G_{0}) > 0$, which implies that $\dd{P}_{\nu}(Q_{1}(0)=0) = \nu(G_{0}) > 0$. This proves that $\rho_{1}^{*}(\xi) < 1$. Similarly, $\rho^{*}_{2}(\xi) < 1$ is proved.\pend

\def\cprime{$'$} \def\cprime{$'$} \def\cprime{$'$} \def\cprime{$'$}
  \def\cprime{$'$} \def\cprime{$'$} \def\cprime{$'$}


\end{document}